\documentclass[11pt]{amsart}
\usepackage[shortalphabetic]{amsrefs}
\usepackage{amsfonts}
\usepackage{amsmath}
\usepackage{amssymb}
\usepackage{amsthm}
\usepackage{pstricks}
\usepackage{pst-node}
\usepackage{pst-plot}

\def\d{\partial}

\def\Gbar{\overline{G}}

\def\Qq{\mathbb{Q}}
\def\Zz{\mathbb{Z}}

\newcommand{\aut}{\mathrm{Aut}}
\newcommand{\ct}{{\operatorname{ct}}}

\newcommand{\codim}{\operatorname{codim}}
\newcommand{\im}{\operatorname{im}}

\newcommand{\rt}{{\operatorname{rt}}}
\newcommand{\M}{\mathcal{M}}
\newcommand{\Mct}{\mathcal{M}^\ct}
\newcommand{\Mrt}{\mathcal{M}^\rt}
\newcommand{\Mbar}{\overline{\mathcal{M}}}

\newcommand{\Ss}{\Sigma}

\def\nr{2pt}

\def\gradius{.3}
\newcommand{\genusvertex}[3]{
  \rput#1{\cnode{\gradius}{#2}}\rput#1{\scriptsize $#3$}
}
\newcommand{\markedpoint}[3]{
  \psset{linewidth=0pt,linecolor=white}
  \rput#1{\cnode{4.5pt}{#2}}
  \rput#1{\scriptsize $#3$}
  \psset{linewidth=.8pt,linecolor=black}
}
\psset{unit=0.6cm}
\psset{arrowsize=4pt, arrowlength=.5}

\newtheorem{proposition}{Proposition}
\newtheorem{conjecture}{Conjecture}

\theoremstyle{definition}
\newtheorem{definition}{Definition}
\newtheorem{example}{Example}
\newtheorem{remark}{Remark}

\author{Stephanie Yang}
\address{Institutionen f\"or Matematik, Kungliga Tekniska H\"ogskolan, 100 44 Stockholm, Sweden}
\email{stpyang@math.kth.se}

\title{Calculating intersection numbers on moduli spaces of curves}

\begin{document}

\begin{abstract}
We discuss an algorithm for calculating intersection numbers for
tautological classes on $\Mbar_{g,n}$, and use this to compute the
coefficients of a genus 4 tautological relation in cohomology whose
existence follows directly from the work of Bergstr\"om-Tommasi.  We
end with the ranks of the graded parts of the Gorenstein quotients of
the tautological rings $R^\ast(\Mbar_{g,n})$, as well as of the
related rings $R^\ast(\Mct_{g,n})$ and $R^\ast(\Mrt_{g,n})$, for low
values of $g$ and $n$.
\end{abstract}

\maketitle

\tableofcontents

\section{Introduction}


Let $\Mbar_{g,n}$ denote the moduli space of Deligne-Mumford stable
curves of genus $g$ with $n$ labeled points.  The tautological rings
$R^*(\Mbar_{g,n})$ comprise the smallest system of $\Qq$-subalgebras
of the Chow rings $A^*(\Mbar_{g,n})$ that is closed under the natural
forgetful and gluing morphisms. These rings include the cotangent
classes $\psi_i$, the Mumford-Morita-Miller classes $\kappa_i$, the
Chern classes $\lambda_i$ of the Hodge bundle, and
topologically-defined boundary classes (for definitions and
properties, see \cites{MR1486986,MR2003030,MR717614}).

Programs for computing top intersection numbers among $\psi$,
$\kappa$, and $\lambda$ classes have been implemented in both Maple
and Macaulay2~\cites{KaLa5,HodgeIntegrals}. However, neither of these
handle boundary classes in a satisfactory manner. An algorithm for
doing so is described below in terms of decorated stable graphs.  This
is based on a formula given a paper by Graber and Pandharipande
(\cite{MR1960923}*{\S A.1}) and has has been implemented in Maple.
This program has been used to calculate the ranks of the Gorenstein
quotients of various tautological rings of moduli spaces of curves
$\Mbar_{g,n}$, $\Mct_{g,n}$, and $\Mrt_{g,n}$ for low values of $g$
and $n$, where $\Mct_{g,n}$ denote the locus of curves in
$\Mbar_{g,n}$ ``of compact type,'' whose dual graph is a tree, and
$\Mrt_{g,n}$ denotes the locus of curves in $\Mbar_{g,n}$ ``with
rational tails,'' with one component of genus $g$.  According to
well-known conjectures (\cites{MR1722541,MR1957060}), the tautological
rings of $\Mbar_{g,n}$, $\Mct_{g,n}$ and $\Mrt_{g,n}$ are Gorenstein
and thus equal to their Gorenstein quotients.

\section{Stable graphs}

Graber and Pandharipande were the first to write down an explicit
multiplication formula for boundary classes in $R^\ast(\Mbar_{g,n})$
(\cite{MR1960923}*{\S A.1}); we begin by adapting their notation.  By
a graph, we mean a connected and undirected graph with allowed
half-edges, multiple edges, and self-edges.  In other words, it is a
sextuple
\begin{equation*}
  (V,H,E,N,g\colon V\to\Zz_{\geq 0}, i\colon H\to H)
\end{equation*}
which satisfies the properties:
\begin{enumerate}
  \item $V$ is a finite set of vertices, with a genus function
  $g\colon V\to\Zz_{\geq 0}$.

  \item $H$ is a finite set of half-edges, and $i$ is a involution
    with labeled fixed points $N$.

  \item $E$ is the set of nontrivial orbits of $i$ and $(V,E)$ defines
    a connected graph.

\end{enumerate}

\noindent To avoid ambiguity, we will sometimes use subscripts
$(V_G,H_G,\ldots,i_G)$ when referring a specific graph $G$.  A graph
is called stable if all vertices satisfy the stability condition
$2g(v)-2+n(v)>0$, where $n\colon V\to\Zz_{\geq 0}$ is the a function
which assigns a vertex $v$ to the total number of half edges in $H$
incident to it.  The total genus of a graph $G$ is
\begin{align*}
  g(G) :&= \sum_{v\in V}g(v) + h^1(G)\\
       &= \sum_{v\in V} g(v) + |E|-|V|+1.
\end{align*}

Any pointed stable curve $C$ has an associated stable graph, called
the {\em dual graph} of $C$, that encodes its topological data. The
dual graph is constructed with the following rules: irreducible
components of $C$ correspond to vertices $V$, nodes of $C$ correspond
to edges $E$, and labeled marked points correspond to labeled
half-edges $N$.  Conversely, given any stable graph $G$, we define
$\sigma_G$ to be the closure of the locus of curves in
$\Mbar_{g(G),|N|}$ whose dual graph is $G$.  Another way to define
$\sigma_G$ is using the composition of gluing morphisms:
\begin{equation}
  \iota_G\colon \prod_{v\in V} \Mbar_{g(v),n(v)}\to\Mbar_{g,n},
\end{equation}
where $(g,n) = (g(G),|N|)$, and points on curves are identified in the
manner prescribed by the graph $G$ (\cite{MR1960923}*{Proposition 8}).
The boundary stratum $\sigma_G$ is then equal to
\begin{equation}
  \sigma_G:= \im(\iota_G)
\end{equation}
The loci of curves with a fixed dual graph form the stratification of
$\Mbar_{g,n}$ by topological type.

A {\em specialization} of a graph $G$ is a graph $H$ which is obtained
by replacing each vertex of $G$ with a graph of genus $g(v)$ with
$n(v)$ half edges that are identified with the half edges incident to
$v$.  This corresponds roughly to specialization of curves.

\begin{definition}\label{def:graphstructure}
A $G$-structure on $A$ is an identification of a specialization of
$G$ with $A$. In other words, it is a triple
\begin{equation*}
  \left(
  \alpha\colon V_A\twoheadrightarrow V_G,
  \beta\colon H_G\hookrightarrow H_A,
  \gamma\colon H_A\setminus \im(\beta) \to V_G
  \right)
\end{equation*}
which satisfies:
\begin{enumerate}
  \item The map $\beta$ commutes with involution
    ($\beta\circ\iota_G=\iota_A\circ\beta$) and induces an isomorphism
    between the fixed points $N_G\stackrel{\sim}{\to}N_A$,
  \item Any half-edge $h\in\im(\beta)$ is incident to $v$ if and only
    if $\beta^{-1}(h)$ is incident to $\alpha(v)$
  \item If $h\in H_A\setminus\im(B)$ is incident to $v$, then
    $\gamma(h)=\alpha(v)$.
  \item If $v\in V_G$, then the preimage
    $(\alpha^{-1}(v),\gamma^{-1}(v)/\iota_A)$ is a connected graph of
    genus $g(v)$.
\end{enumerate}
\end{definition}

\begin{example}\label{ex:manystructures}
There can be many $G$-structures on the same graph $A$. Let $G$ and
$A$ be the graphs denoted in the pictures below
\begin{align*}
  G &= 
  \begin{pspicture}(-.7,-.1)(1,0.2)
    \rput(0,.14){
      \cnode*(0,0){\nr}{v1}
      \markedpoint{(-.5196,0.3)}{mp1}{1}
      \markedpoint{(-.5196,-.3)}{mp2}{2}
      \ncline[arrows=-]{mp1}{v1}
      \ncline[arrows=-]{mp2}{v1}
      \pnode(0.8,0){x1}
      \nccurve[angleA=45,angleB=90,arrows=-]{v1}{x1}
      \nccurve[angleA=-45,angleB=-90]{v1}{x1}
    }
  \end{pspicture}
  &A &=
  \begin{pspicture}(-.7,0)(1.7,0)
    \rput(0,.14){
      \cnode*(0,0){\nr}{v1}
      \cnode*(1,0){\nr}{v2}
      \markedpoint{(-.6,0)}{mp1}{1}
      \markedpoint{(1.6,0)}{mp2}{2}
      \ncline[arrows=-]{mp1}{v1}
      \ncline[arrows=-]{mp2}{v2}
      \nccurve[angleA=35,angleB=155]{v1}{v2}
      \nccurve[angleA=-35,angleB=-155]{v1}{v2}
    }
  \end{pspicture}
\end{align*}
There are four $G$-structures on $A$ that respect the labels of the
half-edges.  The map $\beta$ identifies one of the two edges of $A$
with the edge of $G$ in one of two different ways.
\end{example}

A $(G,H)$-graph is a graph $A$ which has both a $G$-structure and an
$H$-structure, called a $(G,H)$-structure. Two $(G,H)$-structures on a
graph $A$ are considered isomorphic if they differ by an automorphism
of $A$.  If $A$ has a $(G,H)$-structure and $e=(h_1,h_2)$ is an edge
of $A$, then we say $e$ is a {\em common $(G,H)$-edge} if it is
identified with both an edge of $G$ and with an edge of $H$, i.e.,~if
\begin{equation}
   h_1,h_2\in\beta(H_G)\cap\beta(H_H)\subseteq H_A.
\end{equation}
A $(G,H)$-graph $A$ is called {\em generic} if every edge of $A$ is
identified with an edge of $G$, an edge of $H$, or both, i.e.~if
\begin{equation}
  \beta(H_G)\cup\beta(H_H) = H_A.
\end{equation}
The set of all generic $(G,H)$-structures is denoted $\Gamma(G,H)$.

\begin{example}
Let $A$ and $G$ be the same graphs as in
Example~\ref{ex:manystructures}, and set $H=G$. There are sixteen
$(G,H)$-structures on $A$. Eight of them are generic. They are
isomorphic in pairs; i.e., up to an automorphism of $A$ there are only
eight different $(G,H)$-structures, four of which are generic.
\end{example}

\section{Decorated graphs}

A {\em decorated stable graph} $G$ is a stable graph $\overline{G}$
with the additional data of a monomial
\begin{equation}
 \theta_v = \prod_{i=1}^{n(v)}\psi_i^{e_i}\prod_{j=1}^m\kappa_j^{f_j}
\end{equation}
chosen for each vertex $v$. Define $\sigma_G$ to be the tautological
class
\begin{equation}
  \sigma_G:= \frac{1}{|\aut(G)|}\iota_{\Gbar\ast}\left(\prod_{v\in
  V}\theta_v\right).
\end{equation}
Here $\Gbar$ denotes the underlying stable graph of $G$ without any
vertex decorations.

\begin{remark}
  The cotangent $\psi$-classes are indexed by the half-edges incident
  to a vertex, so we denote them by adding arrowheads to the
  appropriate half-edge (Figure~\ref{fig:decorated}).
\end{remark}

\psset{unit=0.9cm}
\psset{arrowsize=4pt, arrowlength=.5}

\begin{figure}
\hspace{\fill}
\begin{pspicture}(0,0)(2,1.6)
  \pscurve[showpoints=false](0,0)(.6,.6)(.6,1)(0,1.5)(0,1)(.6,.6)(1.4,.8)(1.8,1.6)
  \pscurve[showpoints=false](1.2,1.6)(1.3,1.3)(1.4,.8)(1.6,0.4)(2,0)
  \psdots[dotscale=1 1](1.6,0.4)
  \psdots[dotscale=1 1](1.3,1.3)
  \uput[l](1.3,1.3){\scriptsize $1$}
  \uput[l](1.7,0.4){\scriptsize $2,\psi$}
  \uput[u](-.1,-.1){\normalsize $2$}
  \uput[ur](2,-.2){\normalsize $3,\kappa_1$}
\end{pspicture}
\hspace{\fill}
\begin{pspicture}(-.6,-.8)(1.1,1.6)
  \genusvertex{(0,0)}{v1}{2}
  \genusvertex{(1,0)}{v2}{3}
  \markedpoint{(1.0,0.8)}{mp1}{1}
  \markedpoint{(1.8,0)}{mp2}{2}

  \ncline[arrows=-]{mp1}{v2}
  \ncline[arrows=->]{mp2}{v2}
  \ncline[arrows=-]{v1}{v2}
  \pnode(-0.7,0){x1}
  \nccurve[angleA=135,angleB=90]{v1}{x1}
  \nccurve[angleA=-135,angleB=-90]{v1}{x1}
  \uput[ur](1,0.1){\small $\kappa_1$}
\end{pspicture}
\hspace{\fill}
\caption{A $2$-pointed curve of genus 6 and its dual decorated graph}\label{fig:decorated}
\end{figure}
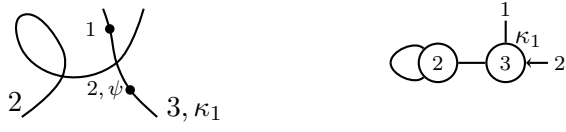

\psset{unit=0.6cm}
\psset{arrowsize=4pt, arrowlength=.5}

Let $\Ss^\ast(\Mbar_{g,n})$ denote the graded vector space of
decorated genus $g$ stable graphs with $n$ labeled half-edges, with
the grading by codimension
\begin{align}
  \codim(G) := |E_G|+\sum_{v\in V}\codim(\theta_v), \\
\intertext{where}
 \codim\left(\prod_{i=1}^n\psi_i^{e_i}\prod_{j=1}^m\kappa_j^{f_j}\right)=\sum_{i=1}^n e_i+\sum_{j=1}^m jf_j.
\end{align}
For any decorated graph $A$, and vertex $v\in V_A$, let
$\kappa_{a,v}\cdot A$ denote the graph where the decoration $\theta_v$
is replaced with $\kappa_a\theta_v$, with associative and distributive
rules
\begin{align}
  (\kappa_{a,v}\kappa_{a',v})\cdot A &=
  \kappa_{a,v}\cdot(\kappa_{a',v}\cdot A), \\
  (\kappa_{a,v}+\kappa_{a',v})\cdot A &=
  \kappa_{a,v}\cdot A + \kappa_{a',v}\cdot A.
\end{align}
Similarly, if $h$ is any half-edge of $A$, then we let $\psi_{h}\cdot
A$ denote the graph $A$ where the half-edge $h$ is decorated with an
additional arrowhead.

\begin{example}
The pullback of the class $\kappa_a$ to a boundary stratum $B$ is the
sum of strata
\begin{equation}
  \kappa_a\sigma_A = \sum_{v\in V(A)} \sigma_{\kappa_{a,v}\cdot A}.
\end{equation}
\end{example}

Let $G$ be a decorated graph, and let $A$ be an undecorated graph with
a $\overline{G}$-structure $(\alpha,\beta,\gamma)$.  Let $v$ be a
vertex in $G$ decorated with the monomial $\theta_v :=
\prod_{i=1}^n\psi_i^{e_i}\prod_{j=1}^m\kappa_j^{f_j}$, and set
\begin{align}
  f_{A}(G,v) &=
  \prod_{i=1}^n \psi_{\beta(i)}^{e_i}
  \prod_{j=1}^m\left(
    \sum_{w\in\alpha^{-1}(v)}\kappa_{j,w}
  \right)^{f_j}, \\
  F_A(G,H)&=\sum_{v \in V_G} f_{A}(G,v)
  \sum_{v \in V_H} f_{A}(H,v)
  \prod_{e = (h_1,h_2)}
  \left(-\psi_{h_1}-\psi_{h_2}\right).
\end{align}
The last product is taken over all common
$(\overline{G},\overline{H})$-edges of $A$.

We define multiplication in $\Sigma^\ast(\Mbar_{g,n})$ with the
formula
\begin{equation}\label{multrule}
  G\cdot H := \sum_{A\in\Gamma(\overline{G},\overline{H})}
  \frac{1}{\left|\aut(A)\right|}
  F_{A}(G,H)\cdot A,
\end{equation}
thus turning $\Sigma^\ast$ into a graded algebra.


\begin{proposition}\label{prop:Phi}
There is a natural surjective map of rings
\begin{equation}
  \Phi\colon\Ss^\ast(\Mbar_{g,n})\to R^\ast(\Mbar_{g,n})
\end{equation}
which sends a decorated dual graph to its associated stratum class.
\end{proposition}

\begin{proof}
  This is a map from \cite{MR1960923}*{Formula 11}, and surjectivity
  follows from \cite{MR1960923}*{Proposition 11}.
\end{proof}

\begin{example}
We calculate the product of $\sigma_G$ and $\sigma_H$, where $G$ and
$H$ are the following decorated stable graphs in
$\Sigma^3(\Mbar_{4,0})$ and $\Sigma^6(\Mbar_{4,0})$.  Note that
vertices of genus 0 are denoted by dots.

\begin{align*}
G &=
\begin{pspicture}(-1,0)(1,1)
  \rput(0,0.14){
    \genusvertex{(0,0)}{v1}{3}
    \pnode(-.8,0){x1}
    \nccurve[angleA=135,angleB=90]{v1}{x1}
    \nccurve[angleA=-135,angleB=-90]{v1}{x1}
    \rput(.4,0.4){$\kappa_2$}
  }
\end{pspicture}
&H &=
\begin{pspicture}(-1,0)(2,1)
  \rput(0,0.14){
    \cnode*(0,.5){\nr}{v1}
    \cnode*(0,-.5){\nr}{v2}
    \cnode*(.5,0){\nr}{v3}
    \genusvertex{(1.5,0)}{v4}{2}
    \pnode(-.57,1.07){x1}
    \pnode(-.57,-1.07){x2}
    \ncline{v1}{v3}
    \ncline{v2}{v3}
    \ncline[arrows=->]{v3}{v4}
    \nccurve[angleA=90,angleB=45]{v1}{x1}
    \nccurve[angleA=180,angleB=-135]{v1}{x1}
    \nccurve[angleA=180,angleB=135]{v2}{x2}
    \nccurve[angleA=-90,angleB=-45]{v2}{x2}
  }
\end{pspicture}
\end{align*}
\smallskip

There are two
generic $(\overline{G},\overline{H})$-graphs:
\begin{align*}
A &=
\begin{pspicture}(-1,0)(2,1)
  \rput(0,0,14){
    \cnode*(0,.5){\nr}{v1}
    \cnode*(0,-.5){\nr}{v2}
    \cnode*(.5,0){\nr}{v3}
    \genusvertex{(1.5,0)}{v4}{2}
    \pnode(-.57,1.07){x1}
    \pnode(-.57,-1.07){x2}
    \ncline{v1}{v3}
    \ncline{v2}{v3}
    \ncline{v3}{v4}
    \nccurve[angleA=90,angleB=45]{v1}{x1}
    \nccurve[angleA=180,angleB=-135]{v1}{x1}
    \nccurve[angleA=180,angleB=135]{v2}{x2}
    \nccurve[angleA=-90,angleB=-45]{v2}{x2}
  }
\end{pspicture}
&B &=
\begin{pspicture}(-1,0)(2,1)
  \rput(0,0,14){
    \cnode*(0,.5){\nr}{v1}
    \cnode*(0,-.5){\nr}{v2}
    \cnode*(.5,0){\nr}{v3}
    \genusvertex{(1.5,0)}{v4}{1}
    \pnode(-.57,1.07){x1}
    \pnode(-.57,-1.07){x2}
    \pnode(2.2,0){x3}
    \ncline{v1}{v3}
    \ncline{v2}{v3}
    \ncline{v3}{v4}
    \nccurve[angleA=90,angleB=45]{v1}{x1}
    \nccurve[angleA=180,angleB=-135]{v1}{x1}
    \nccurve[angleA=180,angleB=135]{v2}{x2}
    \nccurve[angleA=-90,angleB=-45]{v2}{x2}
    \nccurve[angleA=45,angleB=90]{v4}{x3}
    \nccurve[angleA=-45,angleB=-90]{v4}{x3}
  }
\end{pspicture}
\end{align*}

\bigskip

\noindent Graph $A$ has four $\overline{G}$-structures and eight
$\overline{H}$-structures, a total of 32
$(\overline{G},\overline{H})$-structures, each of which has exactly
one common edge.  The order of the automorphism group of $A$ is eight.
Graph $B$ has sixteen generic $(\overline{G},\overline{H})$-structure
with no common edges.  The order of the automorphism group of $B$ is
sixteen.

According to formula~\eqref{multrule}, the product $F_A(G,H)\cdot A$
is the sum of eight graphs
\begin{multline}
  F_A(G,H)\cdot A := 
  -\begin{pspicture}(-.6,0)(2,2)
    \rput(0,0.14){
      \cnode*(0,.5){\nr}{v1}
      \cnode*(0,-.5){\nr}{v2}
      \cnode*(.5,0){\nr}{v3}
      \genusvertex{(1.5,0)}{v4}{2}
      \rput(-.4,.3){\small$\kappa_2$}
      \pnode(-.57,1.07){x1}
      \pnode(-.57,-1.07){x2}
      \ncline{v1}{v3}
      \ncline{v2}{v3}
      \ncline[arrows=->]{v3}{v4}
      \nccurve[arrows=<-,angleA=90,angleB=45]{v1}{x1}
      \nccurve[angleA=180,angleB=-135]{v1}{x1}
      \nccurve[angleA=180,angleB=135]{v2}{x2}
      \nccurve[angleA=-90,angleB=-45]{v2}{x2}
    }
  \end{pspicture}
  -\begin{pspicture}(-.6,0)(2,2)
    \rput(0,0.14){
      \cnode*(0,.5){\nr}{v1}
      \cnode*(0,-.5){\nr}{v2}
      \cnode*(.5,0){\nr}{v3}
      \genusvertex{(1.5,0)}{v4}{2}
      \rput(-.4,-.3){\small$\kappa_2$}
      \pnode(-.57,1.07){x1}
      \pnode(-.57,-1.07){x2}
      \ncline{v1}{v3}
      \ncline{v2}{v3}
      \ncline[arrows=->]{v3}{v4}
      \nccurve[arrows=<-,angleA=90,angleB=45]{v1}{x1}
      \nccurve[angleA=180,angleB=-135]{v1}{x1}
      \nccurve[angleA=180,angleB=135]{v2}{x2}
      \nccurve[angleA=-90,angleB=-45]{v2}{x2}
    }
  \end{pspicture}
  -
  \begin{pspicture}(-.6,0)(2,2)
    \rput(0,0.14){
      \cnode*(0,.5){\nr}{v1}
      \cnode*(0,-.5){\nr}{v2}
      \cnode*(.5,0){\nr}{v3}
      \genusvertex{(1.5,0)}{v4}{2}
      \rput(0,0){\small$\kappa_2$}
      \pnode(-.57,1.07){x1}
      \pnode(-.57,-1.07){x2}
      \ncline{v1}{v3}
      \ncline{v2}{v3}
      \ncline[arrows=->]{v3}{v4}
      \nccurve[arrows=<-,angleA=90,angleB=45]{v1}{x1}
      \nccurve[angleA=180,angleB=-135]{v1}{x1}
      \nccurve[angleA=180,angleB=135]{v2}{x2}
      \nccurve[angleA=-90,angleB=-45]{v2}{x2}
    }
  \end{pspicture}
  -
  \begin{pspicture}(-.6,0)(2,2)
    \rput(0,0.14){
      \cnode*(0,.5){\nr}{v1}
      \cnode*(0,-.5){\nr}{v2}
      \cnode*(.5,0){\nr}{v3}
      \genusvertex{(1.5,0)}{v4}{2}
      \rput(1.3,0.4){\small$\kappa_2$}
      \pnode(-.57,1.07){x1}
      \pnode(-.57,-1.07){x2}
      \ncline{v1}{v3}
      \ncline{v2}{v3}
      \ncline[arrows=->]{v3}{v4}
      \nccurve[arrows=<-,angleA=90,angleB=45]{v1}{x1}
      \nccurve[angleA=180,angleB=-135]{v1}{x1}
      \nccurve[angleA=180,angleB=135]{v2}{x2}
      \nccurve[angleA=-90,angleB=-45]{v2}{x2}
    }
  \end{pspicture}
  \\
  -
  \begin{pspicture}(-.6,0)(2,2)
    \rput(0,0.14){
      \cnode*(0,.5){\nr}{v1}
      \cnode*(0,-.5){\nr}{v2}
      \cnode*(.5,0){\nr}{v3}
      \genusvertex{(1.5,0)}{v4}{2}
      \rput(-.4,.3){\small$\kappa_2$}
      \pnode(-.57,1.07){x1}
      \pnode(-.57,-1.07){x2}
      \ncline{v1}{v3}
      \ncline{v2}{v3}
      \ncline[arrows=->]{v3}{v4}
      \nccurve[angleA=90,angleB=45]{v1}{x1}
      \nccurve[arrows=<-,angleA=180,angleB=-135]{v1}{x1}
      \nccurve[angleA=180,angleB=135]{v2}{x2}
      \nccurve[angleA=-90,angleB=-45]{v2}{x2}
    }
  \end{pspicture}
  -
  \begin{pspicture}(-.6,0)(2,2)
    \rput(0,0.14){
      \cnode*(0,.5){\nr}{v1}
      \cnode*(0,-.5){\nr}{v2}
      \cnode*(.5,0){\nr}{v3}
      \genusvertex{(1.5,0)}{v4}{2}
      \rput(-.4,-.3){\small$\kappa_2$}
      \pnode(-.57,1.07){x1}
      \pnode(-.57,-1.07){x2}
      \ncline{v1}{v3}
      \ncline{v2}{v3}
      \ncline[arrows=->]{v3}{v4}
      \nccurve[angleA=90,angleB=45]{v1}{x1}
      \nccurve[arrows=<-,angleA=180,angleB=-135]{v1}{x1}
      \nccurve[angleA=180,angleB=135]{v2}{x2}
      \nccurve[angleA=-90,angleB=-45]{v2}{x2}
    }
  \end{pspicture}
  -
  \begin{pspicture}(-.6,0)(2,2)
    \rput(0,0.14){
      \cnode*(0,.5){\nr}{v1}
      \cnode*(0,-.5){\nr}{v2}
      \cnode*(.5,0){\nr}{v3}
      \genusvertex{(1.5,0)}{v4}{2}
      \rput(0,0){\small$\kappa_2$}
      \pnode(-.57,1.07){x1}
      \pnode(-.57,-1.07){x2}
      \ncline{v1}{v3}
      \ncline{v2}{v3}
      \ncline[arrows=->]{v3}{v4}
      \nccurve[angleA=90,angleB=45]{v1}{x1}
      \nccurve[arrows=<-,angleA=180,angleB=-135]{v1}{x1}
      \nccurve[angleA=180,angleB=135]{v2}{x2}
      \nccurve[angleA=-90,angleB=-45]{v2}{x2}
    }
  \end{pspicture}
  -
  \begin{pspicture}(-.6,0)(2,2)
    \rput(0,0.14){
      \cnode*(0,.5){\nr}{v1}
      \cnode*(0,-.5){\nr}{v2}
      \cnode*(.5,0){\nr}{v3}
      \genusvertex{(1.5,0)}{v4}{2}
      \rput(1.3,0.4){\small$\kappa_2$}
      \pnode(-.57,1.07){x1}
      \pnode(-.57,-1.07){x2}
      \ncline{v1}{v3}
      \ncline{v2}{v3}
      \ncline[arrows=->]{v3}{v4}
      \nccurve[angleA=90,angleB=45]{v1}{x1}
      \nccurve[arrows=<-,angleA=180,angleB=-135]{v1}{x1}
      \nccurve[angleA=180,angleB=135]{v2}{x2}
      \nccurve[angleA=-90,angleB=-45]{v2}{x2}
    }
  \end{pspicture}
\end{multline}
\medskip

\noindent all of which lie in the kernel of $\Phi$.

The product $F_B(G,H)\cdot B$ is the sum of four graphs
\begin{multline}
  F_B(G,H)\cdot B :=
  \begin{pspicture}(-.6,0)(2.3,1)
    \rput(0,0.14){
      \cnode*(0,.5){\nr}{v1}
      \cnode*(0,-.5){\nr}{v2}
      \cnode*(.5,0){\nr}{v3}
      \genusvertex{(1.5,0)}{v4}{1}
      \rput(-.4,.3){\small$\kappa_2$}
      \pnode(-.57,1.07){x1}
      \pnode(-.57,-1.07){x2}
      \pnode(2.2,0){x3}
      \ncline{v1}{v3}
      \ncline{v2}{v3}
      \ncline[arrows=->]{v3}{v4}
      \nccurve[angleA=90,angleB=45]{v1}{x1}
      \nccurve[angleA=180,angleB=-135]{v1}{x1}
      \nccurve[angleA=180,angleB=135]{v2}{x2}
      \nccurve[angleA=-90,angleB=-45]{v2}{x2}
      \nccurve[angleA=45,angleB=90]{v4}{x3}
      \nccurve[angleA=-45,angleB=-90]{v4}{x3}
    }
  \end{pspicture}
  +
  \begin{pspicture}(-.6,0)(2.3,1)
    \rput(0,0.14){
      \cnode*(0,.5){\nr}{v1}
      \cnode*(0,-.5){\nr}{v2}
      \cnode*(.5,0){\nr}{v3}
      \genusvertex{(1.5,0)}{v4}{1}
      \rput(-.4,-.3){\small$\kappa_2$}
      \pnode(-.57,1.07){x1}
      \pnode(-.57,-1.07){x2}
      \pnode(2.2,0){x3}
      \ncline{v1}{v3}
      \ncline{v2}{v3}
      \ncline[arrows=->]{v3}{v4}
      \nccurve[angleA=90,angleB=45]{v1}{x1}
      \nccurve[angleA=180,angleB=-135]{v1}{x1}
      \nccurve[angleA=180,angleB=135]{v2}{x2}
      \nccurve[angleA=-90,angleB=-45]{v2}{x2}
      \nccurve[angleA=45,angleB=90]{v4}{x3}
      \nccurve[angleA=-45,angleB=-90]{v4}{x3}
    }
  \end{pspicture}
  +
  \begin{pspicture}(-.6,0)(2.3,1)
    \rput(0,0.14){
      \cnode*(0,.5){\nr}{v1}
      \cnode*(0,-.5){\nr}{v2}
      \cnode*(.5,0){\nr}{v3}
      \genusvertex{(1.5,0)}{v4}{1}
      \rput(0,0){\small$\kappa_2$}
      \pnode(-.57,1.07){x1}
      \pnode(-.57,-1.07){x2}
      \pnode(2.2,0){x3}
      \ncline{v1}{v3}
      \ncline{v2}{v3}
      \ncline[arrows=->]{v3}{v4}
      \nccurve[angleA=90,angleB=45]{v1}{x1}
      \nccurve[angleA=180,angleB=-135]{v1}{x1}
      \nccurve[angleA=180,angleB=135]{v2}{x2}
      \nccurve[angleA=-90,angleB=-45]{v2}{x2}
      \nccurve[angleA=45,angleB=90]{v4}{x3}
      \nccurve[angleA=-45,angleB=-90]{v4}{x3}
    }
  \end{pspicture}
  +
  \begin{pspicture}(-.6,0)(2.3,1)
    \rput(0,0.14){
      \cnode*(0,.5){\nr}{v1}
      \cnode*(0,-.5){\nr}{v2}
      \cnode*(.5,0){\nr}{v3}
      \genusvertex{(1.5,0)}{v4}{1}
      \rput(1.3,0.4){\small$\kappa_2$}
      \pnode(-.57,1.07){x1}
      \pnode(-.57,-1.07){x2}
      \pnode(2.2,0){x3}
      \ncline{v1}{v3}
      \ncline{v2}{v3}
      \ncline[arrows=->]{v3}{v4}
      \nccurve[angleA=90,angleB=45]{v1}{x1}
      \nccurve[angleA=180,angleB=-135]{v1}{x1}
      \nccurve[angleA=180,angleB=135]{v2}{x2}
      \nccurve[angleA=-90,angleB=-45]{v2}{x2}
      \nccurve[angleA=45,angleB=90]{v4}{x3}
      \nccurve[angleA=-45,angleB=-90]{v4}{x3}
    }
  \end{pspicture}
\end{multline}
\end{example}

\medskip
\noindent Only the last term above does not lie in the kernel of
$\Phi$, and so
\begin{equation}
  \sigma_G\cdot\sigma_H=\Phi(G\cdot H)
  = \Phi\left(
  \begin{pspicture}(-.6,0)(2.3,1.5)
    \rput(0,0.14){
      \cnode*(0,.5){\nr}{v1}
      \cnode*(0,-.5){\nr}{v2}
      \cnode*(.5,0){\nr}{v3}
      \genusvertex{(1.5,0)}{v4}{1}
      \rput(1.3,0.4){\small$\kappa_2$}
      \pnode(-.57,1.07){x1}
      \pnode(-.57,-1.07){x2}
      \pnode(2.2,0){x3}
      \ncline{v1}{v3}
      \ncline{v2}{v3}
      \ncline[arrows=->]{v3}{v4}
      \nccurve[angleA=90,angleB=45]{v1}{x1}
      \nccurve[angleA=180,angleB=-135]{v1}{x1}
      \nccurve[angleA=180,angleB=135]{v2}{x2}
      \nccurve[angleA=-90,angleB=-45]{v2}{x2}
      \nccurve[angleA=45,angleB=90]{v4}{x3}
      \nccurve[angleA=-45,angleB=-90]{v4}{x3}
    }
  \end{pspicture}
  \right)
  = \int_{\Mbar_{1,3}}\psi_1\kappa_2 =
  \frac{1}{8}.
\end{equation}

\section{Curves of compact type, genus 4 and 5}

Recall from \S 1 that $\Mct_{g,n}$ denotes the locus of curves in
$\Mbar_{g,n}$ of compact type and $\Mrt_{g,n}$ denotes the locus of
curves in $\Mbar_{g,n}$ with rational tails.  Define the tautological
rings $R^\ast(\Mct_{g,n})$ and $R^\ast(\Mrt_{g,n})$ by restriction.
Conjecturally, the rings $R^\ast(\Mbar_{g,n})$, respectively
$R^\ast(\Mct_{g,n})$ and $R^\ast(\Mrt_{g,n})$ are Gorenstein with top
degree $3g-3+n$, respectively $2g-3+n$ and $g-2+n$
(\cites{MR1486986,MR2003030,MR717614}). The one-dimensionality of the
top degrees are known in all cases (\cites{MR1346214,MR1722541,MR2176546}), with the isomorphism
given by evaluation
\begin{equation}
   \alpha \mapsto \int_{\Mbar_{g,n}} \alpha,
\end{equation}
in the case of $\Mbar_{g,n}$, and using the evaluation classes
$\lambda_g$ and $\lambda_g\lambda_{g-1}$ in the cases of $\Mct_{g,n}$
and $\Mrt_{g,n}$,
\begin{align}
   R^{2g-3+n}(\Mct_{g,n})&\cong\Qq
   &R^{g-2+n}(\Mrt_{g,n})&\cong\Qq\\
   \alpha &\mapsto \int_{\Mbar_{g,n}} \alpha\lambda_g
   &\alpha &\mapsto \int_{\Mbar_{g,n}} \alpha\lambda_g\lambda_{g-1}.
\end{align}
In this section we show that $R^\ast(\Mct_4)$ is in fact Gorenstein,
and discuss the structure of $R^\ast(\Mct_5)$.

By~\cite{MR2120989}*{Proposition 2}, any tautological decoration
$\theta_v$ of sufficiently high codimension lies on the tautological
ring of the boundary $R^\ast(\d\Mbar_{g,n})$, defined to be the
subring of $A^\ast(\d\Mbar_{g,n})$ generated by pushforwards of
tautological classes on boundary divisors via gluing morphisms. As a
result, the map $\Phi$ defined in Proposition~\ref{prop:Phi} remains
surjective when restricted to the graphs whose vertex decorations
satisfy
\begin{equation}\label{eq:lowdegree}
  \codim(\theta_v) < g(v)+\delta_{0g(v)}-\delta_{0n(v)},
\end{equation}
and the tautological restriction sequence 
\begin{equation}\label{eq:restrict}
  R^{k-1}(\d\Mbar_{g,n})\longrightarrow R^k(\Mbar_{g,n})\longrightarrow R^k(\M_{g,n})\longrightarrow 0
\end{equation}
is exact in degrees $k\geq g(v)+\delta_{0g(v)}-\delta_{0n(v)}$.  Note
that since $A^k(\d\Mbar_{g,n})=R^k(\d\Mbar_{g,n})$ when $k=0,1$, the
restriction sequence is also exact when $k=1,2$.  The exactness of
this sequence for intermediate values remains
unknown~(\cite{MR2120989}*{Conjecture 2}).

There are 30 graphs of genus $4$ whose decorations satisfy
inequality~\eqref{eq:lowdegree}.  Five relations among the thirty
strata represented by these graphs can be seen as as follows: the
relation $\kappa_1=0$ on $\M_2$ pulls back to $\kappa_1-\psi_1=0$ on
$\M_{2,1}$. By exactness of the restriction
sequence~\eqref{eq:restrict} when $k=1$, this extends to a relation on
$\M_{2,1}$,
\begin{equation}\label{eq:m21}
  \kappa_1- \psi_1 = \frac{7}{5}\delta_G
\end{equation}
where $G$ is the graph
\begin{equation}
  G = 
  \begin{pspicture}(-.3,0.3)
  \rput(0,0.14){
    \genusvertex{(0,0)}{v1}{1}
    \genusvertex{(1,0)}{v2}{1}
    \markedpoint{(-.5,0.5)}{mp1}{1}
    \ncline{mp1}{v1}
    \ncline{v1}{v2}
  }
  \end{pspicture}
\end{equation}
This relation, when pushed forward via various gluing morphisms to
relations on $\Mct_{4}$, allows one to express the strata
corresponding to the following five graphs in terms of the remaining
twenty-five.
\begin{equation}
  \begin{pspicture}(1,0.4)(2,-0.4)
  \rput(0,0.0){
    \genusvertex{(0,0)}{v1}{2}
    \genusvertex{(1,0)}{v2}{2}
    \rput(-.3,0.4){\small$\kappa_1$}
    \ncline{v1}{v2}
  }
  \end{pspicture}
  \begin{pspicture}(-1,0.4)(2,-0.4)
  \rput(0,0.0){
    \genusvertex{(0,0)}{v1}{2}
    \genusvertex{(1,0)}{v2}{2}
    \rput(-.3,0.4){\small$\kappa_1$}
    \ncline[arrows=->]{v1}{v2}
  }
  \end{pspicture}
  \begin{pspicture}(-1,0.4)(2,-0.4)
  \rput(0,0.0){
    \genusvertex{(0,0)}{v1}{2}
    \genusvertex{(1,0)}{v2}{2}
    \rput(-.3,0.4){\small$\kappa_1$}
    \rput(1.3,0.4){\small$\kappa_1$}
    \ncline{v1}{v2}
  }
  \end{pspicture}
  \begin{pspicture}(-1,0.4)(3,-0.4)
  \rput(0,0.0){
    \genusvertex{(0,0)}{v1}{2}
    \genusvertex{(1,0)}{v2}{1}
    \genusvertex{(2,0)}{v3}{1}
    \rput(-.3,0.4){\small$\kappa_1$}
    \ncline{v1}{v2}
    \ncline{v2}{v3}
  }
  \end{pspicture}
  \begin{pspicture}(-1,0.4)(3,-0.4)
  \rput(0,0.0){
    \cnode*(1,0){\nr}{v1}
    \genusvertex{(0,0)}{v2}{2}
    \genusvertex{(2,0.5)}{v3}{1}
    \genusvertex{(2,-.5)}{v4}{1}
    \ncline{v1}{v2}
    \rput(-.3,0.4){\small$\kappa_1$}
    \ncline{v1}{v3}
    \ncline{v1}{v4}
  }
  \end{pspicture}
\end{equation}
Similarly, relation~\eqref{eq:m21} can be pulled back to $\Mct_{2,2}$
and then pushed forward via gluing morphisms, allowing us to eliminate
one more graph:
\begin{equation}
  \begin{pspicture}(-1,0.2)(3,-0.2)
  \rput(0,0.0){
    \genusvertex{(0,0)}{v1}{1}
    \genusvertex{(1,0)}{v2}{2}
    \genusvertex{(2,0)}{v3}{1}
    \rput(1,0.4){\small$\kappa_1$}
    \ncline{v1}{v2}
    \ncline{v2}{v3}
  }
  \end{pspicture}
\end{equation}

In a similar manner we get three additional relations from
$R^2(\M_{3,1})$, which is one-dimensional. (This is the socle
statement of Faber's conjecture, which is proved in
\cite{MR2120989}*{Proposition 3} and \cite{MR2176546}*{\S 5.8}.)
Since~\eqref{eq:restrict} is exact when $k=2$, the relations
$\psi_1\kappa_1=5\psi_1^2$, $\kappa_2=\psi_1^2$, and
$\kappa_1^2=9\psi_1^2$ extend to relations on $\Mct_{3,1}$
\begin{align}
  \psi_1\kappa_1&=5\psi_1^2
  +\frac{16}{21}\sigma_{G_1}
  +\frac{5}{7}\sigma_{G_2}
  +\frac{40}{21}\sigma_{G_3}
  -\frac{61}{21}\sigma_{G_4}
  +\frac{4}{35}\sigma_{G_5}
  -\frac{16}{35}\sigma_{G_6}\\
  \kappa_2&=\psi_1^2
  +\frac{41}{21}\sigma_{G_1}
  +\frac{41}{21}\sigma_{G_3}
  -\frac{41}{21}\sigma_{G_4}
  -\frac{4}{35}\sigma_{G_5}
  -\frac{8}{35}\sigma_{G_6}\\
  \kappa_1^2&=9\psi_1^2
  +\frac{299}{21}\sigma_{G_1}
  +\frac{10}{7}\sigma_{G_2}
  +\frac{347}{21}\sigma_{G_3}
  -\frac{389}{21}\sigma_{G_4}
  +\frac{19}{35}\sigma_{G_5}
  -\frac{2}{7}\sigma_{G_6}
\end{align}
where
\begin{align}
  G_1 &=
  \begin{pspicture}(-.3,-0.1)(1,0.9)
  \rput(0,0.14){
    \genusvertex{(0,0)}{v1}{2}
    \genusvertex{(1,0)}{v2}{1}
    \markedpoint{(-.6,0.6)}{mp1}{1}
    \ncline[arrows=<-]{v1}{v2}
    \ncline{v1}{mp1}
  }
  \end{pspicture}
  &G_2 &=
  \begin{pspicture}(-.3,-0.1)(1,0.9)
  \rput(0,0.14){
    \genusvertex{(0,0)}{v1}{2}
    \genusvertex{(1,0)}{v2}{1}
    \markedpoint{(-.6,0.6)}{mp1}{1}
    \ncline{v1}{v2}
    \ncline[arrows=<-]{v1}{mp1}
  }
  \end{pspicture}
  &G_3 &=
  \begin{pspicture}(-.3,-0.1)(1,0.9)
  \rput(0,0.14){
    \genusvertex{(0,0)}{v1}{2}
    \genusvertex{(1,0)}{v2}{1}
    \markedpoint{(1.6,0.6)}{mp1}{1}
    \ncline[arrows=<-]{v1}{v2}
    \ncline{v2}{mp1}
  }
  \end{pspicture}
  \\
  G_4 &=
  \begin{pspicture}(-.3,-0.1)(1,0.9)
  \rput(0,0.14){
    \genusvertex{(0,0)}{v1}{2}
    \cnode*(1,0){\nr}{v2}
    \genusvertex{(2,0)}{v3}{1}
    \markedpoint{(1,0.8)}{mp1}{1}
    \ncline{v1}{v2}
    \ncline{v2}{v3}
    \ncline{v2}{mp1}
  }
  \end{pspicture}
  &G_5 &=
  \begin{pspicture}(-.3,-0.1)(1,0.9)
  \rput(0,0.14){
    \genusvertex{(0,0)}{v1}{1}
    \genusvertex{(1,0)}{v2}{1}
    \genusvertex{(2,0)}{v3}{1}
    \markedpoint{(1,0.8)}{mp1}{1}
    \ncline{v1}{v2}
    \ncline{v2}{v3}
    \ncline{v2}{mp1}
  }
  \end{pspicture}
  &G_6 &=
  \begin{pspicture}(-.3,-0.1)(1,0.9)
  \rput(0,0.14){
    \genusvertex{(0,0)}{v1}{1}
    \genusvertex{(1,0)}{v2}{1}
    \genusvertex{(2,0)}{v3}{1}
    \markedpoint{(-.6,0.6)}{mp1}{1}
    \ncline{v1}{v2}
    \ncline{v2}{v3}
    \ncline{v1}{mp1}
  }
  \end{pspicture}
\end{align}
The coefficients above were found using a Maple program which
implemented the algorithm described in the previous section.
Specifically, this program calculated the 7 by 10 matrix of
intersection numbers below

\renewcommand{\arraystretch}{1.5}
\begin{center}
\begin{tabular}{c|ccccccc|ccc}
&$\psi_1^2$&$\sigma_{G_1}$&$\sigma_{G_2}$&$\sigma_{G_3}$&$\sigma_{G_4}$&$\sigma_{G_5}$&$\sigma_{G_6}$&$\psi_1\kappa_1$&$\kappa_2$&$\kappa_1^2$\\\hline
$\psi_1^2$&$\frac{31}{70}$&$\frac{21}{10}$&$\frac{7}{10}$&0&0&1&0&$\frac{31}{7}$&$\frac{31}{7}$&$\frac{248}{7}$\\
$\sigma_{G_1}$&$\frac{21}{10}$&$\frac{13}{10}$&$\frac{-1}{10}$&0&1&-1&0&$\frac{21}{5}$&$\frac{14}{5}$&$\frac{91}{5}$\\
$\sigma_{G_2}$&$\frac{7}{10}$&$\frac{-1}{10}$&$\frac{-11}{10}$&1&0&-3&0&$\frac{42}{5}$&$\frac{14}{5}$&$\frac{91}{5}$\\
$\sigma_{G_3}$&0&0&1&$\frac{-7}{10}$&$\frac{-7}{10}$&2&-1&$\frac{14}{5}$&0&$\frac{21}{5}$\\
$\sigma_{G_4}$&0&1&0&$\frac{-7}{10}$&0&1&-1&$\frac{21}{10}$&$\frac{7}{10}$&$\frac{7}{2}$\\
$\sigma_{G_5}$&1&-1&-3&2&1&0&0&3&1&5\\
$\sigma_{G_6}$&0&0&0&-1&-1&0&0&2&0&2
\end{tabular}
\end{center}


\noindent where each entry is $13824$ times the intersection number of
the classes indexing the rows and columns.

The relations above span its
3-dimension kernel.  Because of these three relations, we may also
disregard graphs in $\Mct_4$ whose genus $3$ vertices are decorated
with $\kappa_1\psi_1$, $\kappa_2$, and $\kappa_1^2$.
\begin{equation}
  \begin{pspicture}(-1,0)(2,0)
  \rput(0,0){
    \genusvertex{(0,0)}{v1}{3}
    \genusvertex{(1,0)}{v2}{1}
    \ncline{v1}{v2}
    \rput(-.4,0.4){\small$\kappa_1$}
    \ncline[arrows=<-]{v1}{v2}
  }
  \end{pspicture}
  \begin{pspicture}(-1,0)(2,0)
  \rput(0,0){
    \genusvertex{(0,0)}{v1}{3}
    \genusvertex{(1,0)}{v2}{1}
    \ncline{v1}{v2}
    \rput(-.4,0.4){\small$\kappa_2$}
    \ncline{v1}{v2}
  }
  \end{pspicture}
  \begin{pspicture}(-1,0)(2,0)
  \rput(0,0){
    \genusvertex{(0,0)}{v1}{3}
    \genusvertex{(1,0)}{v2}{1}
    \ncline{v1}{v2}
    \rput(-.4,0.4){\small$\kappa_1^2$}
    \ncline{v1}{v2}
  }
  \end{pspicture}
\end{equation}

The remaining 21 decorated dual graphs are listed below.

\bigskip

\begin{tabular}{cl}
degree 0: &
\parbox[c]{9cm}{
\begin{pspicture}(0,-.8)(0,.8)
  \genusvertex{(0,0)}{v1}{4}
\end{pspicture}
}\\
degree 1: &
\parbox[c]{9cm}{
\begin{pspicture}(0,-.8)(5,.8)
  \rput(0,0){
    \genusvertex{(0,0)}{v1}{4}
    \rput(0.3,0.3){\small $\kappa_1$}
  }
  \rput(1.5,0){
    \genusvertex{(0,0)}{v1}{3}
    \genusvertex{(1,0)}{v2}{1}
    \ncline{v1}{v2}
  }
  \rput(4,0){
    \genusvertex{(0,0)}{v1}{2}
    \genusvertex{(1,0)}{v2}{2}
    \ncline{v1}{v2}
  }
\end{pspicture}
}
\\
degree 2: &
\parbox[c]{9cm}{
\begin{pspicture}(0,-.8)(14.5,.8)
  \rput(0,0){
    \genusvertex{(0,0)}{v1}{4}
    \rput(0.3,0.3){\small $\kappa_2$}
  }
  \rput(1.5,0){
    \genusvertex{(0,0)}{v1}{3}
    \genusvertex{(1,0)}{v2}{1}
    \ncline{v1}{v2}
    \rput(0.3,0.3){\small $\kappa_1$}
  }
  \rput(4,0){
    \genusvertex{(0,0)}{v1}{3}
    \genusvertex{(1,0)}{v2}{1}
    \ncline{v1}{v2}
    \ncline[arrows=<-]{v1}{v2}
  }
  \rput(6.5,0){
    \genusvertex{(0,0)}{v1}{2}
    \genusvertex{(1,0)}{v2}{2}
    \ncline[arrows=<-]{v1}{v2}
  }
  \rput(9,0){
    \genusvertex{(0,0)}{v1}{2}
    \genusvertex{(1,0)}{v2}{1}
    \genusvertex{(2,0)}{v3}{1}
    \ncline{v1}{v2}
    \ncline{v2}{v3}
  }
  \rput(12.5,0){
    \genusvertex{(0,0)}{v1}{1}
    \genusvertex{(1,0)}{v2}{2}
    \genusvertex{(2,0)}{v3}{1}
    \ncline{v1}{v2}
    \ncline{v2}{v3}
  }
\end{pspicture}
}
\\
\raisebox{.6cm}{degree 3:} &
\parbox[c]{9cm}{
\begin{pspicture}(0,-2.8)(10.5,.8)
  \rput(0,0){
    \genusvertex{(0,0)}{v1}{3}
    \genusvertex{(1,0)}{v2}{1}
    \ncline{v1}{v2}
    \ncline[arrows=<<-]{v1}{v2}
  }
  \rput(2.5,0){
    \genusvertex{(0,0)}{v1}{2}
    \genusvertex{(1,0)}{v2}{2}
    \ncline[arrows=<->]{v1}{v2}
  }
  \rput(5,0){
    \genusvertex{(0,0)}{v1}{2}
    \genusvertex{(1,0)}{v2}{1}
    \genusvertex{(2,0)}{v3}{1}
    \ncline[arrows=<-]{v1}{v2}
    \ncline{v2}{v3}
  }
  \rput(8.5,0){
    \genusvertex{(0,0)}{v1}{1}
    \genusvertex{(1,0)}{v2}{2}
    \genusvertex{(2,0)}{v3}{1}
    \ncline[arrows=->]{v1}{v2}
    \ncline{v2}{v3}
  }
  \rput(1,-1){
    \cnode*(1,-.5){\nr}{v1}
    \genusvertex{(0,-.5)}{v2}{2}
    \genusvertex{(2,0)}{v3}{1}
    \genusvertex{(2,-1)}{v4}{1}
    \ncline{v1}{v2}
    \ncline{v1}{v3}
    \ncline{v1}{v4}
  }
  \rput(4.5,-1){
    \genusvertex{(1,-.5)}{v1}{1}
    \genusvertex{(0,-.5)}{v2}{1}
    \genusvertex{(2,0)}{v3}{1}
    \genusvertex{(2,-1)}{v4}{1}
    \ncline{v1}{v2}
    \ncline{v1}{v3}
    \ncline{v1}{v4}
  }
  \rput(8,-1){
    \genusvertex{(0,-.5)}{v1}{1}
    \genusvertex{(1,-.5)}{v2}{1}
    \genusvertex{(2,-.5)}{v3}{1}
    \genusvertex{(3,-.5)}{v4}{1}
    \ncline{v1}{v2}
    \ncline{v2}{v3}
    \ncline{v3}{v4}
  }
\end{pspicture}
}
\\
degree 4: &
\parbox[c]{9cm}{
\begin{pspicture}(0,-1.8)(10,.8)
  \rput(0,0){
    \cnode*(1,-.5){\nr}{v1}
    \genusvertex{(0,-.5)}{v2}{2}
    \genusvertex{(2,0)}{v3}{1}
    \genusvertex{(2,-1)}{v4}{1}
    \ncline[arrows=->]{v1}{v2}
    \ncline{v1}{v3}
    \ncline{v1}{v4}
  }
  \rput(3.5,0){
    \cnode*(1,-.5){\nr}{v1}
    \genusvertex{(0,0)}{v2}{1}
    \genusvertex{(0,-1)}{v3}{1}
    \genusvertex{(2,-.5)}{v4}{1}
    \genusvertex{(3,-.5)}{v5}{1}
    \ncline{v1}{v2}
    \ncline{v1}{v3}
    \ncline{v1}{v4}
    \ncline{v4}{v5}
  }
  \rput(8,0){
    \cnode*(1,-.5){\nr}{v1}
    \genusvertex{(0,0)}{v2}{1}
    \genusvertex{(2,0)}{v3}{1}
    \genusvertex{(0,-1)}{v4}{1}
    \genusvertex{(2,-1)}{v5}{1}
    \ncline{v1}{v2}
    \ncline{v1}{v3}
    \ncline{v1}{v4}
    \ncline{v1}{v5}
  }
\end{pspicture}
}\\
degree 5: &
\parbox[c]{9cm}{
\begin{pspicture}(0,-.8)(3,.8)
  \rput(0,0){
    \cnode*(1,0){\nr}{v1}
    \cnode*(2,0){\nr}{v2}
    \genusvertex{(0,.5)}{v3}{1}
    \genusvertex{(3,.5)}{v4}{1}
    \genusvertex{(0,-.5)}{v5}{1}
    \genusvertex{(3,-.5)}{v6}{1}
    \ncline{v1}{v3}
    \ncline{v1}{v5}
    \ncline{v1}{v2}
    \ncline{v2}{v4}
    \ncline{v2}{v6}
  }
\end{pspicture}
}\\
\end{tabular}

\bigskip

\noindent Note that there are six decorated strata classes in degree
2, and seven in degree 3. With no other obvious dependencies among
these classes, the conjectured Gorenstein condition for
$R^\ast(\Mct_g)$ suggests the existence of a new degree 3
relation~(\cite{MR1957060}).

\begin{proposition}\label{pr:m4relation}
The following new relation holds among classes in $H_{12}(\Mbar_4)$.
\begin{multline}\label{eq:m4relation}
  0 = 
  7
  \left[
  \begin{pspicture}(-.3,0)(1.3,.5)
    \rput(0,0.14){ 
      \genusvertex{(0,0)}{v1}{3}
      \genusvertex{(1,0)}{v2}{1}
      \ncline[arrows=<<-]{v1}{v2}
    }
  \end{pspicture}
  \right]
  -20
  \left[
  \begin{pspicture}(-.3,0)(1.3,.5)
    \rput(0,0.14){ 
      \genusvertex{(0,0)}{v1}{2}
      \genusvertex{(1,0)}{v2}{2}
      \ncline[arrows=<->]{v1}{v2}
    }
  \end{pspicture}
  \right]
  -\frac{35}{3}
  \left[
  \begin{pspicture}(-.3,0)(2.3,.5)
    \rput(0,0.14){ 
      \genusvertex{(0,0)}{v1}{1}
      \genusvertex{(1,0)}{v2}{2}
      \genusvertex{(2,0)}{v3}{1}
      \ncline[arrows=->]{v1}{v2}
      \ncline[arrows=-]{v2}{v3}
    }
  \end{pspicture}
  \right]
  +\frac{106}{3} 
  \left[
  \begin{pspicture}(-.3,0)(2.3,.5)
    \rput(0,0.14){ 
      \genusvertex{(0,0)}{v1}{1}
      \genusvertex{(1,0)}{v2}{1}
      \genusvertex{(2,0)}{v3}{2}
      \ncline[arrows=-]{v1}{v2}
      \ncline[arrows=->]{v2}{v3}
    }
  \end{pspicture}
  \right]
  \\
  -\frac{22}{3}
  \left[
  \begin{pspicture}(-.3,0)(2.3,1.2)
    \rput(0,0.14){ 
      \genusvertex{(0,.5)}{v1}{1}
      \genusvertex{(0,-.5)}{v2}{1}
      \cnode*(1,0){\nr}{v3}
      \genusvertex{(2,0)}{v4}{2}
      \ncline[arrows=-]{v1}{v3}
      \ncline[arrows=-]{v2}{v3}
      \ncline[arrows=-]{v3}{v4}
    }
  \end{pspicture}
  \right]
  +\frac{34}{5}
  \left[
  \begin{pspicture}(-.3,0)(2.3,1.2)
    \rput(0,0.14){ 
      \genusvertex{(0,.5)}{v1}{1}
      \genusvertex{(0,-.5)}{v2}{1}
      \genusvertex{(1,0)}{v3}{1}
      \genusvertex{(2,0)}{v4}{1}
      \ncline[arrows=-]{v1}{v3}
      \ncline[arrows=-]{v2}{v3}
      \ncline[arrows=-]{v3}{v4}
    }
  \end{pspicture}
  \right]
  -8
  \left[
  \begin{pspicture}(-.3,0)(3.3,.8)
    \rput(0,0.14){ 
      \genusvertex{(0,0)}{v1}{1}
      \genusvertex{(1,0)}{v2}{1}
      \genusvertex{(2,0)}{v3}{1}
      \genusvertex{(3,0)}{v4}{1}
      \ncline[arrows=-]{v1}{v2}
      \ncline[arrows=-]{v2}{v3}
      \ncline[arrows=-]{v3}{v4}
    }
  \end{pspicture}
  \right]
  +\frac{7}{12}  
  \left[
  \begin{pspicture}(-.3,0)(0.8,.8)
    \rput(0,0.14){
      \genusvertex{(0,0)}{v1}{3}
      \pnode(0.7,0){x1}
      \nccurve[angleA=45,angleB=90,arrows=<<-]{v1}{x1}
      \nccurve[angleA=-45,angleB=-90]{v1}{x1}
    }
  \end{pspicture}
  \right]
  \\
  -\frac{1}{36}  
  \left[
  \begin{pspicture}(-.8,0)(0.8,.8)
    \rput(0,0.14){
      \genusvertex{(0,0)}{v1}{2}
      \pnode(-.7,0){x1}
      \pnode(0.7,0){x2}
      \nccurve[angleA=135,angleB=90,arrows=<-]{v1}{x1}
      \nccurve[angleA=-135,angleB=-90]{v1}{x1}
      \nccurve[angleA=45,angleB=90]{v1}{x2}
      \nccurve[angleA=-45,angleB=-90]{v1}{x2}
    }
  \end{pspicture}
  \right]
  +\frac{7}{24}  
  \left[
  \begin{pspicture}(-.3,0)(1.8,.8)
    \rput(0,0.14){
      \genusvertex{(0,0)}{v1}{3}
      \rput(0.3,0.4){\small$\kappa_1$}
      \cnode*(1,0){\nr}{v2}
      \pnode(1.7,0){x1}
      \ncline{v1}{v2}
      \nccurve[angleA=45,angleB=90]{v2}{x1}
      \nccurve[angleA=-45,angleB=-90]{v2}{x1}
    }
  \end{pspicture}
  \right]
  -\frac{7}{4}  
  \left[
  \begin{pspicture}(-.3,0)(1.8,.8)
    \rput(0,0.14){
      \genusvertex{(0,0)}{v1}{3}
      \cnode*(1,0){\nr}{v2}
      \pnode(1.7,0){x1}
      \ncline[arrows=<-]{v1}{v2}
      \nccurve[angleA=45,angleB=90]{v2}{x1}
      \nccurve[angleA=-45,angleB=-90]{v2}{x1}
    }
  \end{pspicture}
  \right]
  -\frac{5}{6}  
  \left[
  \begin{pspicture}(-.8,0)(1.8,.8)
    \rput(0,0.14){
      \genusvertex{(0,0)}{v1}{2}
      \genusvertex{(1,0)}{v2}{1}
      \pnode(-.7,0){x1}
      \ncline{v1}{v2}
      \nccurve[angleA=135,angleB=90,arrows=<-]{v1}{x1}
      \nccurve[angleA=-135,angleB=-90]{v1}{x1}
    }
  \end{pspicture}
  \right]
  -\frac{19}{72}  
  \left[
  \begin{pspicture}(-.8,0)(1.3,.8)
    \rput(0,0.14){
      \genusvertex{(0,0)}{v1}{2}
      \genusvertex{(1,0)}{v2}{1}
      \pnode(-.7,0){x1}
      \ncline[arrows=<-]{v1}{v2}
      \nccurve[angleA=135,angleB=90]{v1}{x1}
      \nccurve[angleA=-135,angleB=-90]{v1}{x1}
    }
  \end{pspicture}
  \right]
  \\
  -\frac{65}{18}  
  \left[
  \begin{pspicture}(-.3,0)(1.3,.8)
    \rput(0,0.14){
      \genusvertex{(0,0)}{v1}{2}
      \genusvertex{(1,0)}{v2}{1}
      \nccurve[angleA=45,angleB=135,arrows=<-]{v1}{v2}
      \nccurve[angleA=-45,angleB=-135]{v1}{v2}
    }
  \end{pspicture}
  \right]
  +\frac{73}{36}  
  \left[
  \begin{pspicture}(-.3,0)(1.8,.8)
    \rput(0,0.14){
      \genusvertex{(0,0)}{v1}{2}
      \genusvertex{(1,0)}{v2}{1}
      \pnode(1.7,0){x}
      \ncline[arrows=<-]{v1}{v2}
      \nccurve[angleA=45,angleB=90]{v2}{x}
      \nccurve[angleA=-45,angleB=-90]{v2}{x}
    }
  \end{pspicture}
  \right]
  +\frac{19}{34560}  
  \left[
  \begin{pspicture}(-.8,0)(.7,.9)
    \rput(0,0.14){
      \genusvertex{(0,0)}{v1}{1}
      \pnode(-.7,0){x1}
      \pnode(.5,.5){x2}
      \pnode(.5,-.5){x3}
      \nccurve[angleA=135,angleB=90]{v1}{x1}
      \nccurve[angleA=-135,angleB=-90]{v1}{x1}
      \nccurve[angleA=90,angleB=135]{v1}{x2}
      \nccurve[angleA=0,angleB=-45]{v1}{x2}
      \nccurve[angleA=0,angleB=45]{v1}{x3}
      \nccurve[angleA=-90,angleB=-135]{v1}{x3}
    }
  \end{pspicture}
  \right]
  -\frac{17}{1728}
  \left[
  \begin{pspicture}(-.8,0)(1.8,.8)
    \rput(0,0.14){
      \genusvertex{(0,0)}{v1}{2}
      \cnode*(1,0){\nr}{v2}
      \pnode(-.7,0){x1}
      \pnode(1.7,0){x2}
      \ncline{v1}{v2}
      \nccurve[angleA=135,angleB=90]{v1}{x1}
      \nccurve[angleA=-135,angleB=-90]{v1}{x1}
      \nccurve[angleA=45,angleB=90]{v2}{x2}
      \nccurve[angleA=-45,angleB=-90]{v2}{x2}
    }
  \end{pspicture}
  \right]
  \\
  +\frac{1}{36}
  \left[
  \begin{pspicture}(-.3,0)(1.1,.8)
    \rput(0,0.14){
      \genusvertex{(0,0)}{v1}{2}
      \cnode*(1,0){\nr}{v2}
      \ncline{v1}{v2}
      \nccurve[angleA=45,angleB=135]{v1}{v2}
      \nccurve[angleA=-45,angleB=-135]{v1}{v2}
    }
  \end{pspicture}
  \right]
  -\frac{71}{864}
  \left[
  \begin{pspicture}(-.3,0)(1.8,.8)
    \rput(0,0.14){
      \genusvertex{(0,0)}{v1}{2}
      \cnode*(1,0){\nr}{v2}
      \pnode(1.7,0){x}
      \nccurve[angleA=45,angleB=135]{v1}{v2}
      \nccurve[angleA=-45,angleB=-135]{v1}{v2}
      \nccurve[angleA=45,angleB=90]{v2}{x}
      \nccurve[angleA=-45,angleB=-90]{v2}{x}
    }
  \end{pspicture}
  \right]
  -\frac{37}{864}
  \left[
  \begin{pspicture}(-.3,0)(1.3,1)
    \rput(0,0.14){
      \genusvertex{(0,0)}{v1}{2}
      \cnode*(1,0){\nr}{v2}
      \pnode(1,0.7){x1}
      \pnode(1,-.7){x2}
      \ncline{v1}{v2}
      \nccurve[angleA=45,angleB=0]{v2}{x1}
      \nccurve[angleA=135,angleB=180]{v2}{x1}
      \nccurve[angleA=-135,angleB=180]{v2}{x2}
      \nccurve[angleA=-45,angleB=0]{v2}{x2}
    }
  \end{pspicture}
  \right]
  +\frac{11}{288}
  \left[
  \begin{pspicture}(-.3,0)(1.3,1)
    \rput(0,0.14){
      \genusvertex{(0,0)}{v1}{1}
      \genusvertex{(1,0)}{v2}{1}
      \pnode(1,0.7){x1}
      \pnode(1,-.7){x2}
      \ncline{v1}{v2}
      \nccurve[angleA=45,angleB=0]{v2}{x1}
      \nccurve[angleA=135,angleB=180]{v2}{x1}
      \nccurve[angleA=-135,angleB=180]{v2}{x2}
      \nccurve[angleA=-45,angleB=0]{v2}{x2}
    }
  \end{pspicture}
  \right]
  \\
  +\frac{11}{72}
  \left[
  \begin{pspicture}(-.8,0)(1.3,.8)
    \rput(0,0.14){
      \genusvertex{(0,0)}{v1}{1}
      \genusvertex{(1,0)}{v2}{1}
      \pnode(-.7,0){x1}
      \nccurve[angleA=135,angleB=90]{v1}{x1}
      \nccurve[angleA=-135,angleB=-90]{v1}{x1}
      \nccurve[angleA=45,angleB=135]{v1}{v2}
      \nccurve[angleA=-45,angleB=-135]{v1}{v2}
    }
  \end{pspicture}
  \right]
  -\frac{7}{360}
  \left[
  \begin{pspicture}(-.8,0)(1.8,.8)
    \rput(0,0.14){
      \genusvertex{(0,0)}{v1}{1}
      \genusvertex{(1,0)}{v2}{1}
      \pnode(-.7,0){x1}
      \pnode(1.7,0){x2}
      \ncline{v1}{v2}
      \nccurve[angleA=135,angleB=90]{v1}{x1}
      \nccurve[angleA=-135,angleB=-90]{v1}{x1}
      \nccurve[angleA=45,angleB=90]{v2}{x2}
      \nccurve[angleA=-45,angleB=-90]{v2}{x2}
    }
  \end{pspicture}
  \right]
  -\frac{2}{15}
  \left[
  \begin{pspicture}(-.3,0)(1.3,.8)
    \rput(0,0.14){
      \genusvertex{(0,0)}{v1}{1}
      \genusvertex{(1,0)}{v2}{1}
      \ncline{v1}{v2}
      \nccurve[angleA=45,angleB=135]{v1}{v2}
      \nccurve[angleA=-45,angleB=-135]{v1}{v2}
    }
  \end{pspicture}
  \right]
  +\frac{5}{18}
  \left[
  \begin{pspicture}(-.3,0)(2.8,.8)
    \rput(0,0.14){
      \genusvertex{(0,0)}{v1}{1}
      \genusvertex{(1,0)}{v2}{2}
      \cnode*(2,0){\nr}{v3}
      \pnode(2.7,0){x}
      \ncline{v1}{v2}
      \ncline{v2}{v3}
      \nccurve[angleA=45,angleB=90]{v3}{x}
      \nccurve[angleA=-45,angleB=-90]{v3}{x}
    }
  \end{pspicture}
  \right]
  \\
  +\frac{1}{8}
  \left[
  \begin{pspicture}(-.3,0)(2.8,.8)
    \rput(0,0.14){
      \genusvertex{(0,0)}{v1}{2}
      \genusvertex{(1,0)}{v2}{1}
      \cnode*(2,0){\nr}{v3}
      \pnode(2.7,0){x}
      \ncline{v1}{v2}
      \ncline{v2}{v3}
      \nccurve[angleA=45,angleB=90]{v3}{x}
      \nccurve[angleA=-45,angleB=-90]{v3}{x}
    }
  \end{pspicture}
  \right]
  +\frac{7}{36}
  \left[
  \begin{pspicture}(-.3,0)(2.3,.8)
    \rput(0,0.14){
      \genusvertex{(0,0)}{v1}{2}
      \cnode*(1,0){\nr}{v2}
      \genusvertex{(2,0)}{v3}{1}
      \nccurve[angleA=45,angleB=135]{v1}{v2}
      \nccurve[angleA=-45,angleB=-135]{v1}{v2}
      \ncline{v2}{v3}
    }
  \end{pspicture}
  \right]
  -\frac{1}{18}
  \left[
  \begin{pspicture}(-.3,0)(2.3,.8)
    \rput(0,0.14){
      \genusvertex{(0,0)}{v1}{2}
      \cnode*(1,0){\nr}{v2}
      \genusvertex{(2,0)}{v3}{1}
      \ncline{v1}{v2}
      \nccurve[angleA=45,angleB=135]{v2}{v3}
      \nccurve[angleA=-45,angleB=-135]{v2}{v3}
    }
  \end{pspicture}
  \right]
  -\frac{4}{3}
  \left[
  \begin{pspicture}(-.3,0)(1.8,1.2)
    \rput(0,0.14){  
      \genusvertex{(0,.5)}{v1}{2}
      \genusvertex{(0,-.5)}{v2}{1}
      \cnode*(1,0){\nr}{v3}
      \pnode(1.7,0){x}
      \ncline{v1}{v3}
      \ncline{v2}{v3}
      \nccurve[angleA=45,angleB=90]{v3}{x}
      \nccurve[angleA=-45,angleB=-90]{v3}{x}
    }
  \end{pspicture}
  \right]
  \\
  +\frac{53}{60}
  \left[
  \begin{pspicture}(-.3,0)(1.8,1.2)
    \rput(0,0.14){ 
      \genusvertex{(0,.5)}{v1}{1}
      \genusvertex{(0,-.5)}{v2}{1}
      \genusvertex{(1,0)}{v3}{1}
      \pnode(1.7,0){x}
      \ncline{v1}{v3}
      \ncline{v2}{v3}
      \nccurve[angleA=45,angleB=90]{v3}{x}
      \nccurve[angleA=-45,angleB=-90]{v3}{x}
    }
  \end{pspicture}
  \right]
  -\frac{4}{5}
  \left[
  \begin{pspicture}(-.3,0)(2.8,.8)
    \rput(0,0.14){
      \genusvertex{(0,0)}{v1}{1}
      \genusvertex{(1,0)}{v2}{1}
      \genusvertex{(2,0)}{v3}{1}
      \pnode(2.7,0){x}
      \ncline{v1}{v2}
      \ncline{v2}{v3}
      \nccurve[angleA=45,angleB=90]{v3}{x}
      \nccurve[angleA=-45,angleB=-90]{v3}{x}
    }
  \end{pspicture}
  \right]
  +\frac{83}{30}
  \left[
  \begin{pspicture}(-.3,0)(2.3,.8)
    \rput(0,0.14){
      \genusvertex{(0,0)}{v1}{1}
      \genusvertex{(1,0)}{v2}{1}
      \genusvertex{(2,0)}{v3}{1}
      \pnode(2.7,0){x}
      \ncline{v1}{v2}
      \nccurve[angleA=45,angleB=135]{v2}{v3}
      \nccurve[angleA=-45,angleB=-135]{v2}{v3}
    }
  \end{pspicture}
  \right]
  +\frac{373}{120}
  \left[
  \begin{pspicture}(-.3,0)(2.3,.8)
    \rput(0,0.14){
      \genusvertex{(0,0)}{v1}{1}
      \genusvertex{(1,-.5)}{v2}{1}
      \genusvertex{(2,0)}{v3}{1}
      \ncline{v1}{v2}
      \ncline{v2}{v3}
      \nccurve[angleA=45,angleB=135]{v1}{v3}
    }
  \end{pspicture}
  \right],
\end{multline}
where $\left[G\right]$ denotes the tautological class $\sigma_G$.
\end{proposition}

This relation was shown to hold in ring $R^3(\Mbar_{4})$ by Faber and
Pandharipande, by a different method.  Their work is not yet
published.

\begin{proof}
The cohomology of $\M_4$ and $\Mbar_4$, calculated respectively by
Tommasi and Bergstr\"om-Tommasi in \cite{MR2134272} and
\cite{MR2295510}, is known to be isomorphic to cohomology with compact
support. Dual to the short exact sequence
\begin{equation}
  0
  \rightarrow H^{12}_c(\Mbar_{4})
  \rightarrow H^{12}_c(\d\M_4)
  \rightarrow H^{13}_c(\M_4)
  \rightarrow 0
\end{equation}
is a sequence of Borel-Moore homology groups
\begin{equation}
  0
  \rightarrow H_{13}^{BM}(\M_{4})
  \rightarrow H_{12}^{BM}(\d\M_4)
  \rightarrow H_{12}^{BM}(\Mbar_4)
  \rightarrow 0.
\end{equation}
The first term $H_{13}^{BM}(\M_4)$ is isomorphic to $H^5(\M_4)$ and is
known to be 1-dimensional (\cite{MR2134272}).  The second term is
isomorphic to $H_{12}(\d\M_4)$ and is spanned by the 33 terms in
equation~\eqref{eq:m4relation}.  The kernel of the map $\phi$ in the
equivalent sequence
\begin{equation}
  0
  \rightarrow H^5(\M_4)
  \rightarrow H_{12}(\d\M_4)
  \stackrel{\phi}{\rightarrow} H_{12}(\Mbar_4)
  \rightarrow 0
\end{equation}
is exactly the relation given above.

\end{proof}

A similar analysis can be done for $R^3(\Mct_5)$. There are 31 graphs
in $\Sigma^3(\Mct_5)$, eleven of which can be eliminated using
relations from genus 2, 3, and 4. The rank of the intersection pairing
$R^3 \times R^4$ on the remaining 20 graphs is only 19, which suggests
a new codimension 3 relation in $\Mct_5$.

\begin{conjecture}
The following relation holds in $R^3(\Mct_5)$.
\begin{multline}
  0 = 
  \left[
  \begin{pspicture}(-.3,0)(1.3,.8)
    \rput(0,0.14){ 
      \genusvertex{(0,0)}{v1}{4}
      \genusvertex{(1,0)}{v2}{1}
      \rput(0.3,0.4){\small$\kappa_2$}
      \ncline{v1}{v2}
    }
  \end{pspicture}
  \right]
  -7
  \left[
  \begin{pspicture}(-.3,0)(1.3,.8)
    \rput(0,0.14){ 
      \genusvertex{(0,0)}{v1}{4}
      \genusvertex{(1,0)}{v2}{1}
      \ncline[arrows=<<-]{v1}{v2}
    }
  \end{pspicture}
  \right]
  -30
  \left[
  \begin{pspicture}(-.3,0)(1.3,.8)
    \rput(0,0.14){ 
      \genusvertex{(0,0)}{v1}{3}
      \genusvertex{(1,0)}{v2}{2}
      \rput(0.3,0.4){\small$\kappa_1$}
      \ncline[arrows=->]{v1}{v2}
    }
  \end{pspicture}
  \right]
  +102
  \left[
  \begin{pspicture}(-.3,0)(1.3,.8)
    \rput(0,0.14){ 
      \genusvertex{(0,0)}{v1}{3}
      \genusvertex{(1,0)}{v2}{2}
      \ncline[arrows=<->]{v1}{v2}
    }
  \end{pspicture}
  \right]
  +48
  \left[
  \begin{pspicture}(-.3,0)(1.3,.8)
    \rput(0,0.14){ 
      \genusvertex{(0,0)}{v1}{3}
      \genusvertex{(1,0)}{v2}{2}
      \ncline[arrows=<<-]{v1}{v2}
    }
  \end{pspicture}
  \right]
  \\
  -4
  \left[
  \begin{pspicture}(-.3,0)(2.3,.8)
    \rput(0,0.14){ 
      \genusvertex{(0,0)}{v1}{1}
      \genusvertex{(1,0)}{v2}{3}
      \genusvertex{(2,0)}{v3}{1}
      \rput(1.3,0.4){\small$\kappa_1$}
      \ncline{v1}{v2}
      \ncline{v2}{v3}
    }
  \end{pspicture}
  \right]
  +19
  \left[
  \begin{pspicture}(-.3,0)(2.3,.8)
    \rput(0,0.14){ 
      \genusvertex{(0,0)}{v1}{1}
      \genusvertex{(1,0)}{v2}{3}
      \genusvertex{(2,0)}{v3}{1}
      \ncline[arrows=->]{v1}{v2}
      \ncline{v2}{v3}
    }
  \end{pspicture}
  \right]
  +17
  \left[
  \begin{pspicture}(-.3,0)(2.3,.8)
    \rput(0,0.14){ 
      \genusvertex{(0,0)}{v1}{3}
      \genusvertex{(1,0)}{v2}{1}
      \genusvertex{(2,0)}{v3}{1}
      \rput(0.3,0.4){\small$\kappa_1$}
      \ncline{v1}{v2}
      \ncline{v2}{v3}
    }
  \end{pspicture}
  \right]
  -84
  \left[
  \begin{pspicture}(-.3,0)(2.3,.8)
    \rput(0,0.14){ 
      \genusvertex{(0,0)}{v1}{3}
      \genusvertex{(1,0)}{v2}{1}
      \genusvertex{(2,0)}{v3}{1}
      \ncline[arrows=<-]{v1}{v2}
      \ncline{v2}{v3}
    }
  \end{pspicture}
  \right]
  \\
  +\frac{507}{7}
  \left[
  \begin{pspicture}(-.3,0)(2.3,.8)
    \rput(0,0.14){ 
      \genusvertex{(0,0)}{v1}{2}
      \genusvertex{(1,0)}{v2}{2}
      \genusvertex{(2,0)}{v3}{1}
      \ncline[arrows=<-]{v1}{v2}
      \ncline{v2}{v3}
    }
  \end{pspicture}
  \right]
  -51
  \left[
  \begin{pspicture}(-.3,0)(2.3,.8)
    \rput(0,0.14){ 
      \genusvertex{(0,0)}{v1}{2}
      \genusvertex{(1,0)}{v2}{2}
      \genusvertex{(2,0)}{v3}{1}
      \ncline[arrows=->]{v1}{v2}
      \ncline{v2}{v3}
    }
  \end{pspicture}
  \right]
  +\frac{160}{7}
  \left[
  \begin{pspicture}(-.3,0)(2.3,.8)
    \rput(0,0.14){ 
      \genusvertex{(0,0)}{v1}{2}
      \genusvertex{(1,0)}{v2}{2}
      \genusvertex{(2,0)}{v3}{1}
      \ncline{v1}{v2}
      \ncline[arrows=<-]{v2}{v3}
    }
  \end{pspicture}
  \right]
  +\frac{190}{7}
  \left[
  \begin{pspicture}(-.3,0)(2.3,.8)
    \rput(0,0.14){ 
      \genusvertex{(0,0)}{v1}{2}
      \genusvertex{(1,0)}{v2}{1}
      \genusvertex{(2,0)}{v3}{2}
      \ncline[arrows=<-]{v1}{v2}
      \ncline{v2}{v3}
    }
  \end{pspicture}
  \right]
  \\
  +\frac{63}{5}
  \left[
  \begin{pspicture}(-.3,0)(2.3,1.2)
    \rput(0,0.14){ 
      \genusvertex{(0,.5)}{v1}{3}
      \genusvertex{(0,-.5)}{v2}{1}
      \cnode*(1,0){\nr}{v3}
      \genusvertex{(2,0)}{v4}{1}
      \ncline[arrows=-]{v1}{v3}
      \ncline[arrows=-]{v2}{v3}
      \ncline[arrows=-]{v3}{v4}
    }
  \end{pspicture}
  \right]
  -\frac{400}{7}
  \left[
  \begin{pspicture}(-.3,0)(2.3,1.2)
    \rput(0,0.14){ 
      \genusvertex{(0,.5)}{v1}{2}
      \genusvertex{(0,-.5)}{v2}{2}
      \cnode*(1,0){\nr}{v3}
      \genusvertex{(2,0)}{v4}{1}
      \ncline[arrows=-]{v1}{v3}
      \ncline[arrows=-]{v2}{v3}
      \ncline[arrows=-]{v3}{v4}
    }
  \end{pspicture}
  \right]
  +\frac{44}{7}
  \left[
  \begin{pspicture}(-.3,0)(2.3,1.2)
    \rput(0,0.14){ 
      \genusvertex{(0,.5)}{v1}{1}
      \genusvertex{(0,-.5)}{v2}{1}
      \genusvertex{(1,0)}{v3}{2}
      \genusvertex{(2,0)}{v4}{1}
      \ncline[arrows=-]{v1}{v3}
      \ncline[arrows=-]{v2}{v3}
      \ncline[arrows=-]{v3}{v4}
    }
  \end{pspicture}
  \right]
  -\frac{4}{7}
  \left[
  \begin{pspicture}(-.3,0)(2.3,1.2)
    \rput(0,0.14){ 
      \genusvertex{(0,.5)}{v1}{2}
      \genusvertex{(0,-.5)}{v2}{1}
      \genusvertex{(1,0)}{v3}{1}
      \genusvertex{(2,0)}{v4}{1}
      \ncline[arrows=-]{v1}{v3}
      \ncline[arrows=-]{v2}{v3}
      \ncline[arrows=-]{v3}{v4}
    }
  \end{pspicture}
  \right]
  \\
  -\frac{141}{7}
  \left[
  \begin{pspicture}(-.3,0)(3.3,.8)
    \rput(0,0.14){ 
      \genusvertex{(0,0)}{v1}{1}
      \genusvertex{(1,0)}{v2}{2}
      \genusvertex{(2,0)}{v3}{1}
      \genusvertex{(3,0)}{v4}{1}
      \ncline[arrows=-]{v1}{v2}
      \ncline[arrows=-]{v2}{v3}
      \ncline[arrows=-]{v3}{v4}
    }
  \end{pspicture}
  \right]
  +
  \frac{23}{7}
  \left[
  \begin{pspicture}(-.3,0)(3.3,.8)
    \rput(0,0.14){ 
      \genusvertex{(0,0)}{v1}{2}
      \genusvertex{(1,0)}{v2}{1}
      \genusvertex{(2,0)}{v3}{1}
      \genusvertex{(3,0)}{v4}{1}
      \ncline[arrows=-]{v1}{v2}
      \ncline[arrows=-]{v2}{v3}
      \ncline[arrows=-]{v3}{v4}
    }
  \end{pspicture}
  \right]
\end{multline}
\end{conjecture}

\section{Gorenstein quotients of tautological rings}

Below we display the ranks of the intersection pairing on
$\Mbar_{g,n}$, $\Mct_{g,n}$, and $\Mrt_{g,n}$, respectively. These
were calculated using Proposition 1, where $\Phi$ is restricted to the
finite set of graphs satisfying inequality~\eqref{eq:lowdegree}.

\hspace*{\fill}

\renewcommand{\arraystretch}{1}
\begin{center}
\begin{tabular}{c||cccccccccc}
\multicolumn{11}{c}{Rank of the intersection pairing on $\Mbar_{g,n}$}\\    
\multicolumn{11}{c}{\scriptsize\sc Codimension}\\    
$(g,n)$ &0 &  1 &   2 &   3 &   4 &   5 &   6 &   7 &  8 & 9\\\hline
$(0,3)$ &1 &    &     &     &     &     &     &     &    &  \\ 
$(0,4)$ &1 &  1 &     &     &     &     &     &     &    &  \\ 
$(0,5)$ &1 &  5 &   1 &     &     &     &     &     &    &  \\ 
$(0,6)$ &1 & 16 &  16 &   1 &     &     &     &     &    &  \\ 
$(0,7)$ &1 & 42 & 127 &  42 &   1 &     &     &     &    &  \\
\hline
$(1,1)$ &1 &  1 &     &     &     &     &     &     &    &  \\ 
$(1,2)$ &1 &  2 &   1 &     &     &     &     &     &    &  \\ 
$(1,3)$ &1 &  5 &   5 &   1 &     &     &     &     &    &  \\ 
$(1,4)$ &1 & 12 &  23 &  12 &   1 &     &     &     &    &  \\ 
$(1,5)$ &1 & 27 & 102 & 102 &  27 &   1 &     &     &    &  \\ 
\hline
$(2,0)$ &1 &  2 &   2 &   1 &     &     &     &     &    &  \\ 
$(2,1)$ &1 &  3 &   5 &   3 &   1 &     &     &     &    &  \\ 
$(2,2)$ &1 &  6 &  14 &  14 &   6 &   1 &     &     &    &  \\ 
$(2,3)$ &1 & 12 &  44 &  67 &  44 &  12 &   1 &     &    &  \\ 
$(2,4)$ &1 & 24 & 144 & 333 & 333 & 144 &  24 &   1 &    &  \\
\hline
$(3,0)$ &1 &  3 &   7 &  10 &   7 &   3 &   1 &     &    &  \\ 
$(3,1)$ &1 &  5 &  16 &  29 &  29 &  16 &   5 &   1 &    &  \\ 
$(3,2)$ &1 &  9 &  42 & 104 & 142 & 104 &  42 &   9 &  1 &  \\
\hline
$(4,0)$ &1 &  4 &  13 &  32 &  50 &  50 &  32 &  13 &  4 & 1\\  
\end{tabular}
\end{center}

\hspace*{\fill}

\newpage

\begin{center}
\begin{tabular}{c||cccccccc}
\multicolumn{9}{c}{Rank of the intersection pairing on $\Mct_{g,n}$}\\    
\multicolumn{9}{c}{\scriptsize\sc Codimension}\\    
$(g,n)$ & 0 &  1 &   2 &   3 &   4 &   5 &   6 &   7\\
\hline
$(1,1)$ & 1 &    &     &     &     &     &     &    \\ 
$(1,2)$ & 1 &  1 &     &     &     &     &     &    \\ 
$(1,3)$ & 1 &  4 &   1 &     &     &     &     &    \\ 
$(1,4)$ & 1 & 11 &  11 &   1 &     &     &     &    \\ 
$(1,5)$ & 1 & 26 &  71 &  26 &   1 &     &     &    \\ 
$(1,6)$ & 1 & 57 & 348 & 348 &  57 &   1 &     &    \\ 
\hline
$(2,0)$ & 1 &  1 &     &     &     &     &     &    \\ 
$(2,1)$ & 1 &  2 &   1 &     &     &     &     &    \\ 
$(2,2)$ & 1 &  5 &   5 &   1 &     &     &     &    \\ 
$(2,3)$ & 1 & 11 &  24 &  11 &   1 &     &     &    \\ 
$(2,4)$ & 1 & 23 & 101 & 101 &  23 &   1 &     &    \\ 
$(2,5)$ & 1 & 47 & 384 & 769 & 384 &  47 &   1 &    \\ 
\hline
$(3,0)$ & 1 &  2 &   2 &   1 &     &     &     &    \\ 
$(3,1)$ & 1 &  4 &   7 &   4 &   1 &     &     &    \\ 
$(3,2)$ & 1 &  8 &  24 &  24 &   8 &   1 &     &    \\ 
$(3,3)$ & 1 & 16 &  82 & 144 &  82 &  16 &   1 &    \\ 
$(3,4)$ & 1 & 32 & 274 & 813 & 813 & 274 &  32 &   1\\ 
\hline
$(4,0)$ & 1 &  3 &   6 &   6 &   3 &   1 &     &    \\ 
$(4,1)$ & 1 &  5 &  17 &  25 &  17 &   5 &   1 &    \\ 
$(4,2)$ & 1 & 10 &  51 & 120 & 120 &  51 &  10 &   1\\ 
\hline
$(5,0)$ & 1 &  3 &  10 &  19 &  19 &  10 &   3 &   1\\ 
\end{tabular}
\end{center}

\begin{center}
\begin{tabular}{c||ccccccccc}
\multicolumn{8}{c}{Rank of the intersection pairing on $\Mrt_{g,n}$}\\    
\multicolumn{8}{c}{\scriptsize\sc Degree}\\    
$(g,n)$ & 0 &  1 &   2 &   3 &   4 &   5 &   6 &  7 &  8  \\
\hline
$(2,0)$ & 1 &    &     &     &     &     &     &    &     \\ 
$(2,1)$ & 1 &  1 &     &     &     &     &     &    &     \\ 
$(2,2)$ & 1 &  3 &   1 &     &     &     &     &    &     \\ 
$(2,3)$ & 1 &  7 &   7 &   1 &     &     &     &    &     \\ 
$(2,4)$ & 1 & 15 &  35 &  15 &   1 &     &     &    &     \\ 
$(2,5)$ & 1 & 31 & 147 & 147 &  31 &   1 &     &    &     \\ 
$(2,6)$ & 1 & 63 & 556 &1126 & 556 &  63 &   1 &    &     \\ 
\hline
$(3,0)$ & 1 &  1 &     &     &     &     &     &    &     \\ 
$(3,1)$ & 1 &  2 &   1 &     &     &     &     &    &     \\ 
$(3,2)$ & 1 &  4 &   4 &   1 &     &     &     &    &     \\ 
$(3,3)$ & 1 &  8 &  15 &   8 &   1 &     &     &    &     \\ 
$(3,4)$ & 1 & 16 &  54 &  54 &  16 &   1 &     &    &     \\ 
$(3,5)$ & 1 & 32 & 188 & 333 & 188 &  32 &   1 &    &     \\ 
\hline
$(4,0)$ & 1 &  1 &   1 &     &     &     &     &    &     \\ 
$(4,1)$ & 1 &  2 &   2 &   1 &     &     &     &    &     \\ 
$(4,2)$ & 1 &  4 &   6 &   4 &   1 &     &     &    &     \\ 
$(4,3)$ & 1 &  8 &  19 &  19 &   8 &   1 &     &    &     \\ 
$(4,4)$ & 1 & 16 &  61 &  95 &  61 &  16 &   1 &    &     \\ 
$(4,5)$ & 1 & 32 & 199 & 470 & 470 & 199 &  32 &  1 &     \\ 
\hline
$(5,0)$ & 1 &  1 &   1 &   1 &     &     &     &    &     \\ 
$(5,1)$ & 1 &  2 &   3 &   2 &   1 &     &     &    &     \\ 
$(5,2)$ & 1 &  4 &   8 &   8 &   4 &   1 &     &    &     \\ 
$(5,3)$ & 1 &  8 &  22 &  33 &  22 &   8 &   1 &    &     \\ 
$(5,4)$ & 1 & 16 &  65 & 136 & 136 &  65 &  16 &  1 &     \\ 
$(5,5)$ & 1 & 32 & 204 & 577 & 852 & 577 & 204 & 32 &   1 \\ 
\hline 
$(6,0)$ & 1 &  1 &   2 &   1 &   1 &     &     &    &     \\ 
$(6,1)$ & 1 &  2 &   4 &   4 &   2 &   1 &     &    &     \\ 
$(6,2)$ & 1 &  4 &   9 &  13 &   9 &   4 &   1 &    &     \\ 
$(6,3)$ & 1 &  8 &  23 &  44 &  44 &  23 &   8 &  1 &     \\ 
$(6,4)$ & 1 & 16 &  66 & 159 & 226 & 159 &  66 & 16 &  1  \\ 
\hline 
$(7,0)$ & 1 &  1 &   2 &   2 &   1 &   1 &     &    &     \\ 
$(7,1)$ & 1 &  2 &   4 &   5 &   4 &   2 &   1 &    &     \\ 
\end{tabular}
\end{center}

\end{document}